\documentclass[10pt]{amsart}
\usepackage{enumerate,amsmath,amssymb,latexsym,
amsfonts, amsthm, amscd}

\theoremstyle{plain}

\newtheorem{theorem}{Theorem}[section]

\numberwithin{equation}{section}

\newcommand{\ra}{\rightarrow}

\addtocounter{section}{-1}

\begin{document}

\title {Automorphisms of Liouville Structures}

\date{}

\author[P.L. Robinson]{P.L. Robinson}

\address{Department of Mathematics \\ University of Florida \\ Gainesville FL 32611  USA }

\email[]{paulr@ufl.edu}

\subjclass{} \keywords{}

\begin{abstract}

By a {\it Liouville structure} on a symplectic manifold $(M, \omega)$ we mean a choice of symplectic potential: that is, a choice of one-form $\theta$ on $M$ such that ${\rm d} \theta = \omega$. We determine precisely all the automorphisms of a Liouville structure in case $(M, \omega)$ is a symplectic vector space and $\theta$ differs from its canonical symplectic potential by the differential of a homogeneous monomial.  

\end{abstract}

\maketitle

\bigbreak

\section{Introduction}

A {\it Liouville form} on the smooth manifold $M$ is a one-form $\theta \in \Omega^1 (M)$ of which the exterior derivative $\omega : = {\rm d} \theta$ is nonsingular at each point; thus, $(M, \omega)$ is an exact symplectic manifold for which $\theta$ is a preferred symplectic potential. The corresponding {\it Liouville field} on $M$ is the unique vector field $\zeta \in {\rm Vec} (M)$ with contraction $\zeta \lrcorner \; \omega = \theta$; the `magic' Cartan formula shows that $L_{\zeta} \omega = \omega$ and the corresponding (Liouville) flow $\phi$ of $\zeta$ satisfies $\phi_t^* \omega = e^t \omega$. We may (indeed, shall) regard the symplectic manifold $(M, \omega)$ as fixed, in which case the Liouville structure is determined by the Liouville field $\zeta$ and the Liouville form $\theta = \zeta \lrcorner \; \omega$ equally. By an automorphism of $(M, \theta)$ we mean a diffeomorphism $g : M \ra M$ such that $g^* \theta = \theta$; the group of all such automorphisms will be denoted by ${\rm Aut} (M, \theta)$. Naturally, automorphisms of $(M, \theta)$ preserve not only $\theta$ and $\omega$ but also $\zeta$ and its flow. 
  
\medbreak

As a special case, we consider the symplectic manifold $(M, \omega)$ that arises from a symplectic vector space $(V, \Omega)$: thus, $M$ is $V$ as a manifold, so that $V$ is canonically isomorphic to each of its tangent spaces $T_zV$ via $V \ra T_z V : v \mapsto v_z$ with $v_z \psi : = \psi '_z (v) = \frac{\rm d}{{\rm d} t}\psi (z + t v)|_{t = 0}$ for each smooth real function $\psi$ on $V$; also, $\omega$ is $\Omega$ transported via these canonical isomorphisms, so that $\omega_z (x_z, y_z) = \Omega(x, y)$ whenever $x, y, z \in V$. We denote by ${\rm Sp} (V, \Omega)$ the linear symplectic group comprising all linear automorphisms $g$ of $V$ such that $\Omega (g x, g y) = \Omega (x, y)$ for all $x, y \in V$ and by ${\rm Sp} (V, \omega)$ the symplectomorphism group comprising all diffeomorphisms $g$ of $V$ such that $g^* \omega = \omega$; of course, ${\rm Sp} (V, \Omega)$ is a subgroup of ${\rm Sp} (V, \omega)$. 

\medbreak 

In this special case, there exists a canonical Liouville structure. 

\begin{theorem} \label{theta}
$(V, \omega)$ carries a unique symplectic potential $\theta^0$ that is invariant under the linear symplectic group ${\rm Sp} (V, \Omega)$: explicitly, if $z \in V$ and $v \in V$ then
	\[\theta^0_z (v_z) = \frac{1}{2} \Omega (z, v). 
\] 
\end{theorem}

\begin{proof} 
Verification that $\theta^0$ so defined is ${\rm Sp} (V, \Omega)$-invariant and satisfies ${\rm d} \theta^0 = \omega$ is an elementary exercise. An arbitrary symplectic potential $\theta$ for $\omega$ differs from $\theta^0$ by a form that is closed and therefore exact: say $\theta = \theta^0 + {\rm d}\psi$ for some $\psi: V \ra \mathbb{R}$. Each $g \in {\rm Sp} (V, \Omega)$ preserves ${\rm d} \psi = \theta - \theta^0$ so that ${\rm d} (\psi \circ g) = {\rm d} (g^* \psi) = g^* ({\rm d} \psi) = {\rm d} \psi$ and therefore $\psi \circ g - \psi$ is constant, with value $0$ as $g$ fixes the origin; the action of ${\rm Sp} (V, \Omega)$ on $V \setminus \{0\}$ being transitive, it follows that $\psi$ is constant and therefore that $\theta = \theta^0$. 

\end{proof}

In what follows, we shall routinely suppress the superscript $0$ and write simply $\theta$ for this canonical Liouville form. The corresponding Liouville field $\zeta$ is precisely one-half the Euler field: if $z \in V$ then its value at $z$ is $\frac{1}{2}z$ made tangent at $z$ so that $\zeta_z = (\frac{1}{2}z)_z$; its (Liouville) flow is given by $\phi_t (z) = e^{\frac{1}{2} t}z$. In terms of symplectic coordinates $(p_1, \dots , p_m, q_1, \dots , q_m)$ this canonical Liouville structure is familiar: 
	\[\theta = \frac{1}{2} \sum_{i = 1}^{m} (p_i {\rm d} q_i - q_i {\rm d} p_i),
\]
	\[\zeta = \frac{1}{2} \sum_{i = 1}^{m} (p_i \frac{\partial}{\partial p_i} + q_i \frac{\partial}{\partial q_i}). 
\]

\medbreak 

This canonical Liouville form is better than invariant under ${\rm Sp} (V, \Omega)$: its invariance under a smooth map $g : V \ra V$ (not assumed to be a diffeomorphism) forces such $g$ to lie in ${\rm Sp} (V, \Omega)$. 

\begin{theorem} \label{Omega}
Let $g: V \ra V$ be a smooth map. If $g^* \theta = \theta$ then $g \in {\rm Sp} (V, \Omega)$. 
\end{theorem} 

\begin{proof} 
The smooth map $g$ preserves the symplectic form, for $g^* \omega = g^* {\rm d} \theta = {\rm d} g^* \theta = {\rm d} \theta = \omega$; accordingly, each derivative $g'_z : V \ra V$ lies in ${\rm Sp} (V, \Omega)$. If also $v \in V$ then 
	\[\theta_z (v_z) = (g^* \theta)_z (v_z) = \theta_{g(z)} (g_* (v_z)) = \theta_{g(z)} ((g'_z v)_{g(z)}) 
\]
thus (after doubling) 
	\[\Omega (z, v) = \Omega (g(z), g'_z v) = \Omega ((g'_z)^{-1} g(z), v)
\]
so nonsingularity of $\Omega$ yields
	\[z = (g'_z)^{-1} g(z) 
\]
or 
	\[g'_z (z) = g(z). 
\]
This implies that $g$ is homogeneous of degree one, preserving the Liouville flow: 
	\[g(e^{\frac{1}{2} t}z) = e^{\frac{1}{2} t}g(z) 
\]
from which (as $t \ra - \infty$) we deduce that $g(0) = 0$. Finally, as $g$ is differentiable at $0$, 
	\[g'_0 (z) = \lim_{s \ra 0} \frac{g(s z) - g(0)}{s} = g(z)
\]
whence $g = g'_0 \in {\rm Sp} (V, \Omega)$. 
\end{proof} 

We may extract from Theorem \ref{theta} and Theorem \ref{Omega} the following complete description of the automorphism group of the canonical $(V, \theta)$. 

\begin{theorem} \label{aut}

\fbox{${\rm Aut} (V, \theta) = {\rm Sp} (V, \Omega)$}

\end{theorem}

  Our aim in this paper is to determine the automorphism group ${\rm Aut} (V, \theta + {\rm d}\psi)$ for a class of elementary functions $\psi$; recall that each Liouville form for $(V, \omega)$ has the form $\theta + {\rm d}\psi$ for some $\psi : V \ra \mathbb{R}$. To be explicit, fix $a \in V$ and fix a positive integer $n$: a homogeneous monomial $\psi: V \ra \mathbb{R}$ is then defined by the rule 
	\[\psi (z) = \frac{1}{2 n} \Omega (a, z)^n 
\]
for all $z \in V$. We shall precisely determine ${\rm Aut} (V, \theta + {\rm d}\psi)$ for each point $a$ and each degree $n$; the cases $n =1$, $n = 2$ and $n \geqslant 3$ will be handled separately. We round off our account with a complete determination of the isomorphisms between these Liouville structures. 

\medbreak

\section{Linear Monomials}

Fix $a \in V$ and define $\psi^a : V \ra \mathbb{R}$ by 
	\[\psi^a (z) = \frac{1}{2} \Omega (a, z). 
\]
The exterior derivative of $\psi^a$ is given by 
	\[{\rm d} \psi^a_z (v_z) = \frac{1}{2} \Omega (a, v)
\]
whence the Liouville form $\theta^a = \theta + {\rm d} \psi^a$ is given by 
	\[\theta^a_z (v_z) = \frac{1}{2} \Omega (z + a, v)
\]
and the corresponding Liouville field $\zeta^a$ by 
	\[\zeta^a_z = [\frac{1}{2} (z + a)]_z 
\]
with Liouville flow 
	\[\phi^a_t (z) = e^{\frac{1}{2} t} \{z + a\} - a. 
\]

\medbreak 

Now, let $g \in {\rm Aut} (V, \theta^a)$. As $g$ also preserves $\zeta^a$ and its Liouville flow, it follows that 
	\[g(e^{\frac{1}{2} t} \{z + a\} - a) = e^{\frac{1}{2} t} \{g(z) + a\} - a
\]
or with $s = e^{\frac{1}{2} t}$
	\[g(s(z + a) - a) = s(g(z) + a) - a. 
\]
Let $s \ra 0$ to deduce that $g$ fixes $-a$: 
	\[g(-a) = -a.
\]
Rearrange to obtain 
	\[g(z) + a = \frac{1}{s} \{ g(s(z + a) - a) - g(-a) \}
\]
and let $s \ra 0$ once more to deduce that 
	\[g(z) + a = g'_{-a} (z + a)  
\]
where $g'_{-a} \in {\rm Sp} (V, \Omega)$ because $g$ preserves $\omega$. We conclude that each $g \in {\rm Aut} (V, \theta^a)$ has the form 
	\[g(z) = \gamma (z + a) - a
\]
for some $\gamma \in {\rm Sp} (V, \Omega)$; conversely, each $g: V \ra V$ having this form is readily verified to be an automorphism of $(V, \theta^a)$. We summarize these findings as follows. 

\begin{theorem} \label{linear}
If $a \in V$ and $\psi^a : V \ra \mathbb{R}: z \mapsto \frac{1}{2} \Omega (a, z)$ then ${\rm Aut} (V, \theta^a)$ comprises precisely all $g: V \ra V$ having the form 
	\[g(z) = \gamma (z + a) - a
\]
for some $\gamma \in {\rm Sp} (V, \Omega)$. 
\end{theorem} 
\begin{flushright}
$\Box$
\end{flushright} 
\medbreak 

Thus, ${\rm Aut} (V, \theta^a)$ comprises precisely all the affine symplectic automorphisms of $(V, \Omega)$ that fix $-a$. 
 
\medbreak 

\section{Quadratic Monomials}

Fix $a \in V$ and define $\psi^a : V \ra \mathbb{R}$ by 
	\[\psi^a (z) = \frac{1}{4} \Omega (a, z)^2. 
\]
The exterior derivative of $\psi^a$ is given by 
	\[{\rm d} \psi^a_z (v_z) = \frac{1}{2} \Omega (a, z) \Omega (a, v)
\]
whence the Liouville form $\theta^a = \theta + {\rm d} \psi^a$ is given by 
	\[\theta^a_z (v_z) = \frac{1}{2} \Omega (z + \Omega (a, z) a, v)
\]
and the corresponding Liouville field $\zeta^a$ by 
	\[\zeta^a_z = [\frac{1}{2} (z + \Omega (a, z)a)]_z 
\]
with Liouville flow 
	\[\phi^a_t (z) = e^{\frac{1}{2} t} \{ z + \frac{1}{2} t \Omega (a, z) a \}. 
\]

\medbreak

Now, let $g \in {\rm Aut} (V, \theta^a)$. As $g$ preserves the time $2 t$ Liouville flow, 
	\[g(e^t \{ z + t \Omega (a, z) a\}) = e^t \{ g(z) + t \Omega (a, g(z)) a \} 
\]
whence letting $t \ra - \infty$ shows that $g$ fixes $0$: 
	\[g(0) = 0. 
\]
Differentiate along the flow and divide by $e^t$ to obtain 
	\[g'_{z_t} (z + (1 +t) \Omega (a, z) a) = g(z) + (1 + t) \Omega (a, g(z)) a
\]
or
	\[g(z) - g'_{z_t} (z) = (1 + t) \{ \Omega (a, z) g'_{z_t} (a) - \Omega (a, g(z)) a\}   
\]
where $z_t = e^t \{ z + t \Omega (a, z) a\}$. Let $t \ra - \infty$ and note that $z_t \ra 0$ so $g'_{z_t} \ra g'_0$; now 
	\[g(z) - g'_{z_t} (z) \ra g(z) - g'_0 (z)
\]
and 
	\[\Omega (a, z) g'_{z_t} (a) - \Omega (a, g(z)) a \ra \Omega (a, z) g'_0 (a) - \Omega (a, g(z)) a 
\]
whence the presence of $1 + t$ in the last equation above forces
	\[\Omega (a, z) g'_0 (a) = \Omega (a, g(z)) a
\]
and therefore
	\[g(z) - g'_{z_t} (z) = (1 + t) \Omega (a, z) \{ g'_{z_t} (z) - g'_0 (z) \}. 
\]
The derivative $g'$ is locally Lipschitz and $z_t = e^t \{ z + t \Omega (a, z) a\}$ converges to $0$ exponentially fast, so $g'_{z_t} (z)- g'_0 (z)$ converges to $0$ exponentially fast, dominating the factor $1 + t$; thus 
	\[g(z) = g'_0 (z) 
\]
and so $g = g'_0 \in {\rm Sp} (V, \Omega)$. Further, $\Omega (a, z) g'_0 (a) = \Omega (a, g(z)) a$ shows that $g(a) = g'_0 (a) = \lambda a$ for some real $\lambda$; now 
	\[\Omega (a, z) \lambda = \Omega (a, g(z)) = \Omega (g^{-1} (a), z) = \Omega (\lambda^{-1} a, z) = \lambda^{-1} \Omega (a, z)
\]
thus $\lambda^2 = 1$ and so $g(a) = \pm a$. Conversely, it is readily verified that each element of ${\rm Sp} (V, \Omega)$ sending $a$ to $\pm a$ is an automorphism of $(V, \theta^a)$. The following is a summary of our findings. 

\begin{theorem} \label{quadratic}
If $a \in V$ and $\psi^a : V \ra \mathbb{R}: z \mapsto \frac{1}{4} \Omega (a, z)^2$ then
	\[{\rm Aut} (V, \theta^a) = \{ g \in {\rm Sp} (V, \Omega) : g(a) = \pm a \}. 
\]
\end{theorem} 
\begin{flushright}
$\Box$
\end{flushright} 
\medbreak 

Note that the sign ambiguity is to be expected, for $\psi^a = \psi^{-a}$ here. Our analysis of $\theta + {\rm d}\psi^a$ applies to $\theta - {\rm d}\psi^a$ in parallel fashion, yielding the same automorphisms. 

\medbreak 

\section{Higher-degree Monomials}

Let $n$ be a positive integer. Fix $a \in V$ and define $\psi^a : V \ra \mathbb{R}$ by 
	\[\psi^a (z) = \frac{1}{2(n + 2)} \Omega (a, z)^{n + 2}.
\]
The exterior derivative of $\psi^a$ is given by 
	\[{\rm d} \psi^a_z (v_z) = \frac{1}{2} \Omega (a, z)^{n + 1} \Omega (a, v)
\]
whence the Liouville form $\theta^a = \theta + {\rm d} \psi^a$ is given by 
	\[\theta^a_z (v_z) = \frac{1}{2} \Omega (z + \Omega (a, z)^{n + 1} a, v)
\]
and the corresponding Liouville field $\zeta^a$ by 
	\[\zeta^a_z = [\frac{1}{2} (z + \Omega (a, z)^{n + 1}a)]_z 
\]
with Liouville flow 
	\[\phi^a_t (z) = e^{\frac{1}{2} t} \{ z  - \frac{1}{n} \Omega (a, z)^{n + 1} a \} + \frac{1}{n} e^{\frac{1}{2} (n + 1)t} \Omega (a, z)^{n + 1} a. 
\]

\medbreak

Our approach to these higher-degree cases is different. We begin by introducing 
	\[f_a : V \ra V : z \mapsto z + \frac{1}{n} \Omega( a, z)^{n + 1}a. 
\]
This is a diffeomorphism: indeed, its inverse maps $z$ to $z - \frac{1}{n} \Omega( a, z)^{n + 1}a$. Further, we claim that $f_a$ pulls $\theta^a$ back to the canonical $\theta = \theta^0$: 
	\[f_a^* \theta^a = \theta. 
\]
To see this, note that 
	\[(f_a)'_z (v) = v + \frac{n + 1}{n} \Omega (a, z)^n \Omega (a, v) a
\]
whereupon  
	\[(f_a^* \theta^a)_z (v_z) = (\theta^a)_{f_a (z)} ((f_a)'_z (v)) = \frac{1}{2} \Omega (f_a (z), (f_a)'_z (v)) = \frac{1}{2} \Omega (z, v) = \theta_z (v_z)
\]
follows by substitution and cancellation. As a consequence, the map $f_a$ is a symplectomorphism: $f_a \in {\rm Sp} (V, \omega)$. 

\begin{theorem} \label{higher}
If $a \in V$ and $\psi^a : V \ra \mathbb{R}: z \mapsto \frac{1}{2(n + 2)} \Omega (a, z)^{n + 2}$ then 
	\[{\rm Aut} (V, \theta^a) = \{ f_a \circ \gamma \circ f_a^{-1} : \gamma \in {\rm Sp} (V, \Omega) \}. 
\]
\end{theorem} 

\begin{proof} 
The polynomial $f_a$ carries us neatly back to basics: the self-map $g$ of $V$ pulls back $\theta^a$ to itself precisely when $f_a^{-1} \circ g \circ f_a$ pulls back $f_a^* \theta^a = \theta$ to itself precisely when $f_a^{-1} \circ g \circ f_a$ lies in ${\rm Sp} (V, \Omega)$ on account of Theorem \ref{aut}.  
\end{proof} 

It is perhaps worth remarking that this was not our original proof for Theorem \ref{higher}. We originally considered a cubic monomial, assuming $g$ to preserve the time $2 t$ Liouville flow: 
	\[g( e^t \{ z - \Omega (a, z)^2 a \} + e^{2t} \Omega (a, z)^2 a) = e^t \{ g(z) - \Omega (a, g(z))^2 a \} + e^{2t} \Omega (a, g(z))^2 a. 
\]
Application of $e^{-t} \frac{\rm d}{{\rm d} t}$ results in 
	\[g'_{z_t} (z - \Omega (a, z)^2 a + 2 e^t \Omega (a, z)^2 a) = g(z) - \Omega (a, g(z))^2 a + 2 e^t \Omega (a, g(z))^2 a
\]
where now $z_t = e^t \{ z - \Omega (a, z)^2 a \} + e^{2t} \Omega (a, z)^2 a$. Rearrangement of the limit as $t \ra - \infty$ yields 
	\[g(z) = g'_0 (z - \Omega (a, z)^2 a) + \Omega (a, g(z))^2 a
\]
from which the observation 
	\[\Omega (a, g(z)) = \Omega (a, g'_0 (z - \Omega (a, z)^2 a))
\]
produces 
\[g(z) = g'_0 (z - \Omega (a, z)^2 a) + \Omega (a, g'_0 (z - \Omega (a, z)^2 a))^2 a. 
\] 
Of course, this quartic $g$ is exactly the composite $f_a \circ g'_0 \circ f_a^{-1}$ in the present case. 

\medbreak 

Among the features that distinguish these higher degrees from the quadratic and the linear is the following. Recall that in Theorem \ref{linear} each automorphism of $(V, \theta^a)$ fixes the point $-a$; recall also that in Theorem \ref{quadratic} each automorphism of $(V, \theta^a)$ fixes the set $\{ a, -a \}$. By contrast, in Theorem \ref{higher} the group ${\rm Aut} (V, \theta^a)$ acts transitively on $V\setminus \{ 0 \}$ because ${\rm Sp} (V, \Omega)$ itself does so.  

\medbreak 

\section{Isomorphisms} 

The determination of all isomorphisms between these Liouville structures can be effected by essentially the same arguments as those employed for the automorphisms; indeed, the subject could have been presented from the viewpoint of isomorphisms. When $a, b \in V$ we shall write simply ${\rm Iso} (V^a, V^b)$ for the set of all isomorphisms from $(V, \theta^a)$ to $(V, \theta^b)$: that is, the set of all $g : V \ra V$ such that $g^* \theta^b = \theta^a$.  

\medbreak 

First, consider the {\it linear} case surrounding Theorem \ref{linear}. When $a \in V$ let  
	\[\tau_a : V \ra V : z \mapsto z + a
\]
denote translation by $a$; a routine calculation establishes the identity 
	\[\tau_a^* \theta = \theta^a. 
\]
In this linear case, if also $b \in V$ then 
	\[{\rm Iso} (V^a, V^b) = \{\tau_b^{-1} \circ \gamma \circ \tau_a : \gamma \in {\rm Sp} (V, \Omega) \}.
\]
This may of course be established by the line of argument used for Theorem \ref{linear}. It also succumbs to the line used for Theorem \ref{higher}: indeed, the self-map $g$ of $V$ satisfies $g^* \theta^b = \theta^a$ precisely when $g^* \tau_b^* \theta = \tau_a^* \theta$ precisely when $(\tau_b \circ g \circ \tau_a^{-1})^* \theta = \theta$ precisely when $\tau_b \circ g \circ \tau_a^{-1} \in {\rm Sp} (V, \Omega)$ by Theorem \ref{aut}. In particular, $\tau_b^{-1} \circ \tau_a$ is a distinguished isomorphism from $(V, \theta^a)$ to $(V, \theta^b)$ and these Liouville structures are all isomorphic to each other. 

\medbreak 

Secondly, consider the {\it quadratic} case surrounding Theorem \ref{quadratic}. In this case, if $a, b \in V$ then 
	\[{\rm Iso} (V^a, V^b) = \{ \gamma \in {\rm Sp} (V, \Omega) : \gamma (a) = \pm b \}. 
\]
This may be established essentially as was Theorem \ref{quadratic}: start from preservation of the respective time $2 t$ Liouville flows, thus 
\[g(e^t \{ z + t \Omega (a, z) a\}) = e^t \{ g(z) + t \Omega (b, g(z)) b \} ; 
\]
then apply $e^{-t} \frac{\rm d}{{\rm d} t}$ and pass to the limit as $t \ra - \infty$. The action of ${\rm Sp} (V, \Omega)$ on $V\setminus \{ 0 \}$ being transitive, it follows that the (quadratic) Liouville structures $(V, \theta^a)$ with $a$ nonzero are all isomorphic to each other; of course, they are {\it not} isomorphic to the canonical Liouville structure $(V, \theta^0)$. Similar comments apply to the Liouville structures $(V, \theta - {\rm d}\psi^a)$ in this quadratic case, with the further remark that if $a$ is nonzero then $(V, \theta - {\rm d}\psi^a)$ and $(V, \theta + {\rm d}\psi^a)$ are {\it not} isomorphic; an isomorphism between them would lead not to the equation $\lambda^2 = 1$ that arose in the proof of Theorem \ref{quadratic} but rather to the equation $\lambda^2 = -1$ which has no real solution. 

\medbreak 

Lastly, consider the {\it higher} cases surrounding Theorem \ref{higher}. In these cases, if $a, b \in V$ then 
\[{\rm Iso} (V^a, V^b) = \{ f_b \circ \gamma \circ f_a^{-1} : \gamma \in {\rm Sp} (V, \Omega) \} 
\] 
as follows easily by the line of argument for Theorem \ref{higher} itself. In particular, $f_b \circ f_a^{-1}$ is a distinguished isomorphism from $(V, \theta^a)$ to $(V, \theta^b)$ and these higher Liouville structures are all isomorphic to each other.  

\medbreak 

Thus, the homogeneous monomial Liouville structures on $(V, \omega)$ that we have considered fall into two isomorphism classes. The one comprises all the Liouville structures $(V, \theta^a)$ associated to quadratic $\psi^a$ as $a \in V$ runs over the nonzero vectors. The other comprises all the Liouville structures $(V, \theta^a)$ associated to linear, cubic and higher-degree monomials as $a$ runs over the whole of $V$; this class contains the canonical Liouville structure $(V, \theta^0)$ which is associated to $\psi^0 = 0$ in any degree. This list ignores the quadratic structures $\{ (V, \theta - {\rm d}\psi^a) : 0 \neq a \in V \}$ which received passing mention; these constitute a separate isomorphism class.  

\medbreak 

\begin{center} 
{\small A}{\footnotesize CKNOWLEDGEMENT}
\end{center} 

\medbreak 

The author is happy to acknowledge Mike Spivak: his Clever Observation on page 399 of [1] saw service in the proof of Theorem \ref{Omega} and elsewhere. 

\medbreak 

\begin{center} 
{\small R}{\footnotesize EFERENCE}
\end{center} 

\medbreak 

[1] M. Spivak, A Comprehensive Introduction to Differential Geometry, Volume Two. Publish or Perish (1979). 

\end{document}